\documentclass[12pt]{amsart}

\usepackage[english]{babel}

\usepackage{amsmath,amsthm,amsfonts,amssymb, mathrsfs, mathbbol}
\usepackage[foot]{amsaddr}
\usepackage{mathtools}
\usepackage{color}

\usepackage{graphics}
\usepackage{booktabs} 
\usepackage{calligra}

\def\dim{n}
\def\dimt{3}
\def\R{{\mathbb R}}

\def\Rd{\R^\dim}
\def\Rdt{\R^3}
\def\Z{{\mathbb Z}}
\def\N{{\mathbb N}}

\def\div{{\rm div}}
\def\dps{\displaystyle}
\def\pt{\partial}

\renewcommand{\vec}[1]{{{#1}}}

\def\x{{\vec{x}}}
\def\y{{\vec{y}}}
\def\U{{\vec{U}}}
\def\vphi{{\vec{\phi}}}
\def\vtheta{{\vec{\theta}}}

\def\rhr{\langle x\rangle}
\def\rhry{\langle y\rangle}

\def\D{\mathscr{D}}
\def\Dp{\D'}
\def\S{\mathscr{S}}

\def\vv{{\vec{v}}}
\def\w{{\vec{w}}}

\def\z{{\vec{z}}}

\def\e{{\vec{e}}}
\def\f{{\vec{f}}}

\def\u{{\vec{u}}}
\def\vv{{\vec{v}}}

\def\vphi{{\vec{\phi}}}

\def\zero{{\vec{0}}}
\def\h{{\vec{h}}}
\def\grad{{\nabla}}
\def\bve{\;|\;}

\def\vvom{\vec{\omega}}
\def\uom{\widehat{\vvom}}

\def\rhr{\langle x \rangle}
\def\fleche{\longrightarrow}
\def\om{\Omega}
\def\<{\langle}
\def\>{\rangle}

\def\text{\mbox}

\def\div{\mathrm{div}\,}

\def\dps{\displaystyle}

\def\R{{\mathbb R}}
\def\Z{{\mathbb Z}}
\def\N{{\mathbb N}}

\def\S{{\mathbb S}}

\def\D{{\mathscr D}}

\def\Dp{{\mathscr D}^{\prime}}

\def\RR{ {\mathscr  R}}
\def\C{{\mathscr C}}

\def\pt{\partial}
\def\om{\Omega}

\def\<{\langle}
\def\>{\rangle}
\def\inclus{{\hookrightarrow}}

\def\grad{\nabla}

\def\rhr{\langle x\rangle}

\def\Rd{\R^n}

\def\bve{\,|\,}

\def\RR{{\mathscr R}}
\def\Rw{\vec{\RR}}

\def\VV{V}

\def\F{f}
\def\FF{F}

\def\eps{\varepsilon}
\def\Ball{{\mathscr B}}
\def\ps{p^\star}

\def\dimm{3}
\def\Rdt{\R^3}
\def\BB{{\mathscr B}}
\def\talpha{\tilde{\alpha}}
\renewcommand{\leq}{\leqslant}
\renewcommand{\geq}{\geqslant}

\usepackage{fancyhdr}  
\usepackage{lipsum}

\title{On  Navier-Stokes equations arising from the rotation of an obstacle in a fluid}

\author[T. Z. Boulmezaoud]{Tahar Z. BOULMEZAOUD$^{1}$}
\address{\rm  $^1$ Universit\'e Paris-Saclay, UVSQ, LMV, Versailles, France.}
\email{tahar.boulmezaoud@uvsq.fr}

\author[N.  Kerdid]{Nabil KERDID$^{3}$}
\address{\rm  $^3$ College of Sciences, Imam University, Riyadh, Kingdom of Saudi Arabia. }
\email{{\tt nabil\_kerdid@yahoo.fr}}

\author[A.  Kourta]{Amel  KOURTA$^{4}$}
\address{\rm  $^4$ Departement of Mathematics, University Mentouri, Constantine, Algeria.}
\email{\tt a.kourta@yahoo.fr}

\newtheorem{theorem}{Theorem}[section]
\newtheorem{proposition}[theorem]{Proposition}
\newtheorem{corollary}[theorem]{Corollary}
\newtheorem{lemma}[theorem]{Lemma}
\newtheorem{remark}[theorem]{Remark}


\begin{document}

\begin{abstract}
We consider the modified Navier–Stokes equations in $\R^3$ describing the motion of a fluid in the presence of a rotating rigid body. 
Weighted Sobolev spaces are used to describe the behavior of solutions at large distances. Under suitable assumptions, w
e prove the existence and regularity of solutions satisfying appropriate conditions at infinity.
 \end{abstract}

\let\MakeUppercase\relax % this disables uppercasing authors
\maketitle

\section{Introduction}
Consider a rigid body $\BB$, modeled as a closed and bounded subset of $\R^3$, moving within an incompressible viscous fluid that occupies the entire space. 
By adopting a reference frame attached to the body, the motion of the fluid is governed by the following equations:
\begin{equation}\label{modifiedNS0}
\begin{array}{rcl}
-\nu\Delta \U  + {\U.\grad \U}.  +  {\vv_\infty}.\grad \U + (\vvom  \times \x).\grad \U -  \vvom \times \U
+\nabla P &=& \F \mbox{ in } \om, \\
\div \U&=&0  \mbox{ in } \om, 
\end{array}
\end{equation}
where $\om = \Rdt \backslash \BB$,  $\U$ and $P$ denote  respectively the fluid velocity  and the pressure, $\vv_\infty$ is the prescribed translational velocity, and  $\vvom$ denotes the
 angular velocity. This system is often supplemented with a boundary condition of the form
$$
\U = \U_\star \mbox{ on } \pt \om. 
$$
When $\vv_\infty = \zero$ and $\vvom  = \zero$, the system reduces to the classical steady Navier-Stokes system. When  $\vv_\infty \ne \zero$ and $\vvom  = \zero$, neglecting the nonlinear term yields the Oseen system.  \\
Over the past two decades, considerable attention has been devoted to the analysis of system (\ref{modifiedNS0}), with particular emphasis on the asymptotic behavior of its solutions. In \cite{silvestre2004}, the author proves the existence of solutions $(\U, P) \in H^2_{loc}(\om) \times  H^1_{loc}(\om)$ such that $\grad u \in H^1(\om)^3$ and $\grad P \in L^2(\om)^3$. In \cite{galdi2003}, system  (\ref{modifiedNS0}) is investigated under the
 assumption that the external force   $\F$ is of the form $\F = \div \FF$ and satisfies the following estimates at large distances: 
$$
 |\F(\x)|  \lesssim \frac{1}{|\x|^3},  \;  |\div \F(\x)|  \lesssim \frac{1}{|\x|^4} \mbox{ and }   |\FF(\x)| \lesssim \frac{1}{|\x|^2}. 
$$
In this case, the author establishes the existence of a solution  $(\u, p)$ satisfying the asymptotic estimates:
$$
 |\u(\x)| \lesssim \frac{1}{|\x|}, \;  |\grad \u(\x)| \lesssim \frac{1}{|\x|^2}, \; |p(\x)| \lesssim \frac{1}{|\x|^2},  \, |\grad p(\x)| \lesssim \frac{1}{|\x|^3}.
$$

In \cite{farwig6}, the authors adopt a different functional framework and extend Galdi's existence result to Lorentz spaces.

Many authors have investigated the existence of weak solutions to equation \eqref{modifiedNS0}. One of the main difficulties in its analysis lies 
in the description of the asymptotic behavior of solutions at large distances, that is, as   $|\x| \fleche +\infty$.  Over the last decades, several approaches have been proposed to address these difficulties.
 One of these approaches consists in seeking solutions in weighted Sobolev spaces. In other words, one looks for solutions satisfying
\begin{equation}\label{integra_cond1}
\int_{\Rdt} (|\x|^2 +1)^k |\U(\x)|^2 < + \infty, \int_{\Rdt} (|\x|^2 +1)^{k+1} |\grad \U(\x)|^2 < + \infty, 
\end{equation}
and
\begin{equation}\label{integra_cond2}
\int_{\Rdt} (|\x|^2 +1)^{k+1}  |P (\x)|^2 < + \infty. 
\end{equation}
This approach was applied with a significant success to Stokes equations, that is when the nonlinear term $\U.\grad\U$ is dropped and  when $\vv_\infty = \vvom = \zero$; see, e. g., \cite{cattabriga1},  \cite{girault3},  \cite{alliot_these},   \cite{girault2},  \cite{farwig3})
\cite{boulmezaoud1},  \cite{boulmezaoud7} and  \cite{boulmezaoud8}. It was also applied to linearized  Oseen's system  when the motion of
   the body is only translational, $\vv_\infty\ne \zero$ and $\vvom= \zero$ (see  \cite{farwig3}, \cite{boulmezaoud5},   \cite{farwig2}, \cite{kracmar},  \cite{galdi1, galdi2} and references therein). However, 
   few results are available concerning nonlinear equations (\ref{modifiedNS0}) when the body is rotating ($\vvom \ne \zero$). 
   In \cite{farwig7, farwig8}, authors  prove existence of  weak solutions $(\u, p) \in L^1_{loc}(\Rdt)^\dim \times L^1_{loc}(\Rdt)$ satisfying 
\begin{equation}\label{estimateLq}
\|   \varrho  \grad^2 \u  \|_{L^q(\Rdt)^{\dim^3}} + \| \varrho \grad p\|_{L^q(\Rdt)^\dim}  < +\infty, 
\end{equation}
for some appropriate weight functions $\varrho$. Hishida \cite{hishida2006}  deals with solution satisfying 
\begin{equation}
\|\grad \u\|_{L^q(\Rdt)^\dim} + \|p\|_{L^q(\Rdt)} < +\infty.
\end{equation}
This result was generalized by Farwig and Hishida \cite{farwig6}, where existence was established in Lorentz spaces. In \cite{abada_boulmezaoud}, the authors proved the existence of weak solutions to the linearized equations: 
\begin{equation}\label{linearmodifiedNS}
\begin{array}{rcl}
- \nu \Delta \U  - (\vvom \times \x). \grad \U + \vvom \times \U  + \grad P  &=& \F \mbox{ in } \R^3, \\
\div \U &=& 0 \mbox{ in } \R^3, 
\end{array}
\end{equation}
satisfying conditions of the form (\ref{integra_cond1}) and (\ref{integra_cond2}) when 
$-2 \leq k \leq 2$.  
In this paper we  deal with the following non-linear system  in the whole space
\begin{equation}\label{modifiedNS2}
\begin{array}{rcl}
- \nu \Delta \U  - (\vvom \times \x). \grad \U + \vvom \times \U + \U . \grad \U + \grad P  &=& \F \mbox{ in } \R^3, \\
\div \U &=& 0 \mbox{ in } \R^3, 
\end{array}
\end{equation}
completed  with asymptotic conditions of the form  (\ref{integra_cond1}) and (\ref{integra_cond2}).  
Here, we will prove that under certain conditions on the right-hand side $\F$, 
 the system  \eqref{modifiedNS2} admits solutions satisfying
 \begin{equation}
\lim_{|\x| \fleche +\infty} |\x|^{2-\eps} \|\U(|\x|, .)\|_{L^p(\S^2)} = 0, 
\end{equation}
\begin{equation}
\lim_{|\x| \fleche +\infty} |\x|^{3-\eps} \|P(|\x|, .)\|_{L^p(\S^2)} = 0,
\end{equation} 
with  $p  > 3/2$ and  $0 < \eps <1$. Here $\S^2$ denotes the unit sphere of $\Rdt$. Setting
$$
\u(\x)= \frac{\lambda}{\nu}  \U (\frac{\x}{\lambda}), \; \pi(\x)= \frac{\lambda^2}{\nu^2}  P(\frac{\x}{\lambda}), \; \f(\x) =  \frac{\lambda^3}{\nu^2}  \F (\frac{\x}{\lambda}),   \; \mbox{ with } \lambda = \left(\frac{|\vvom|}{|\nu|}\right)^{1/2},
 $$
 we can rewrite system (\ref{modifiedNS2}) without loss  into the normalized form
\begin{equation}\label{prob_opT}
\begin{array}{rcl}
-\Delta \u + \u . \grad \u  -\left(\uom\times \x \right)\nabla \u +\uom\times \u +\nabla \pi &=& \f(\x), \; \mbox{ in } \Rdt, \\
\div \u &=& 0, 
\end{array}
\end{equation}
where 
$$
\uom = \frac{\vvom}{|\vvom|}. 
$$
The outline of this paper is as follows. In section \ref{main_res}  we state our main result concerning existence
and uniqueness of solutions to equations (\ref{prob_opT}) under  adequate assympotions. Section \ref{proof_mr_sec} contains
 some useful results concerning weighted spaces in $\Rd$, $n$ being arbitrary.  The last section \ref{proof_mr_sec} is devoted to
 the proof of the main result. 

\section{The main result}\label{main_res}
We first introduce some notations. In what follows, $C_b(\R^n)$,  $\dim \geq 1$ being an integer,  designates the space of bounded continuous functions  on $\R^2$ endowed with the sup norm.  Given an extended real $p$, $1 \leq p \leq +\infty$, we denote by $p'$ its conjugate defined  by
$\frac{1}{p} + \frac{1}{p'}= 1.$ As usual,  $L^{p}(\Rd)$ (resp. $L^\infty(\Rd)$)  is   the  space of (equivalence classes of) all measurable functions that
 are $p^{th}$ power integrable (resp. essentially bounded)  on $\R^{\dim}$. 
For $\x \in \R^\dim$, set 
 $$
 \rhr = (|\x|^2+1)^{1/2}. 
 $$
 Given two  integers $m \geq 0$ and $k \in \Z$ define   the space 
  \begin{equation*}
 W^{m,p}_{k}(\Rd)=\{u\in \Dp(\Rd) \bve \forall \mu \in \N^{\dim},
 |\mu|\leq m , \rhr^{|\mu|-m+k}D^{\mu}u\in L^{p}(\Rd)\},
 \end{equation*}
 endowed  with the norm
 $$
\left\Vert u\right\Vert _{W_{k }^{m,p}(\Rd) }=(
\sum_{|\mu| \leq m} \left\Vert \rhr
^{\left\vert \mu \right\vert -m+k }D^{\mu }u\right\Vert
_{L^p(\Rd)}^{p}) ^{\frac{1}{p}}
$$
We write $W^{-m,p'}_{-k}(\Rd)$  the dual space of $W^{m,p}_{k}(\Rd)$. The following  algebraic and topological inclusions hold
 $$...\subset W^{m,p}_{k}(\Rd)\subset W^{m-1,p}_{k-1}(\Rd)\subset ...\subset
W^{0,p}_{k-m}(\Rd) \subset W^{-1,p}_{k-m-1}(\Rd)  \subset ...$$
H\"older inequality also gives the imbedding
\begin{equation}\label{holder_weighted}
 W^{0, p}_\alpha(\Rd) \inclus W^{0, q}_\beta(\Rd)
\end{equation}

 when $1 \leq q < p \leq +\infty$ and $\alpha  - \beta   > \dim/q -  \dim/p$ (the imbedding obviously holds when $q=p$ and $\alpha \geq \beta$).     
 Note also that the following Green's formula holds  
 for  $u \in W^{1, p}_\alpha(\Rd)$ and $v \in W^{1, p'}_{-\alpha+1}(\Rd)$           % c corrigé
 \begin{equation}
 \int_{\Rd} \frac{\pt u}{\pt x_i} v dx =  - \int_{\Rd} u \frac{\pt v}{\pt x_i}  dx. 
 \end{equation}
 For further details on the spaces  $W^{m,p}_k$, the reader may consult the following references, which are by no means exhaustive:   \cite{hanouzet},  \cite{avantaggiati},  \cite{kufner}, \cite{kufner2}, \cite{giroire} and \cite{boulmezaoud2}. \\
 $\;$\\
 Although Problem (\ref{prob_opT}) considered here is formulated in a three-dimensional setting, some preliminary results of this paper will be stated in 
$\R^n$, where $n$ is not necessarily equal to 3. Nevertheless, Problem  (\ref{prob_opT})  will be treated exclusively in the case $n=3$.  Accordingly, when 
$n=3$,  the vectorial operator $\Rw$ defined on vector fields of $\R^{\dimm}$ as
 $$
 \Rw \vv = - \uom\wedge \vv + (\uom \wedge \x).\nabla \vv, 
$$
(see, e. g.,  \cite{abada_boulmezaoud} for some basic  properties of $\Rw$). \\
For  $m, k  \in \Z$ and $ 1 \leq p \leq +\infty$, define   the  weighted  space 
$$ \VV^{m, q}_k(\Rdt) = \{ \vv \in W^{m, q}_k(\Rdt)^\dimt \bve \Rw\vv  \in W^{m-2, q}_{k}(\Rdt)^\dimm\}.$$ 
This space is equipped with the norm
$$ \|\vv\|_{\VV^{m, p}_k(\Rdt)} = \dps{ \left\{ \|\vv\|^p_{ W^{m,
p}_k(\Rdt) } +\| \Rw \vv \|^p_{ W^{m-2, p}_{k}(\Rdt)}  \right\}^{1/p}}. $$ 
$\;$\\
A pair  $(\u, \pi)$  in  $\VV^{1,p}_{k}(\Rdt)^\dimt \times W^{0,p}_{k}(\Rdt)$ is said to be a weak solution of (\ref{prob_opT}) 
 if it satisfies this system in the sense of distributions. When $k > 2 - \frac{3}{p}$, it can be shown that $(\u, \pi) \in \VV^{1,p}_{k}(\Rdt)^\dimt \times W^{0,p}_{k}(\Rdt)$ is a weak solution of   (\ref{prob_opT}) if and only if:  
\begin{equation}\label{condi_gradgrad}
\int_{\Rdt}\grad \u .\grad \vphi  dx  + \< \u . \grad \u,  \vphi \>   + \< \Rw \u, \vphi\>  -\int_{\Rdt} \pi \div \vphi dx = \< \f, \vphi\>,  
\end{equation}
for all $ \vphi \in W^{1, p'}_{-k}(\Rdt)^3$. Here $\<.,.\>$ denotes the duality product between $W^{1, p}_k(\Rdt)$ and $W^{-1, p'}_{-k}(\Rdt)$.  The proof 
of this statement relies on the facts that   $\D(\Rdt)$ is dense in $W^{1, p'}_{-k}(\Rdt)$ (see, e. g., \cite{hanouzet}) and that 
 $\u . \grad \u \in W^{-1,p}_{k}(\Rdt)^3$ when $\u \in W^{1, p}_k(\Rdt)^3$ and $ k > 2 - \frac{3}{p}$. \\
  The main result of the paper is this
\begin{theorem}\label{main_result1}
Let  $p > 3/2$, $p \ne 3$, and  set
 $$
 k = [\frac{\dimt}{p'}] =
 \left\{
 \begin{array}{ll}
 1 &\mbox{ if } \frac{3}{2} < p < 3, \\
  2 &\mbox{ if } p > 3. \\
 \end{array}
 \right. 
  $$
Then, there exist two constants $\kappa > 0$ and $c > 0$ such that for any $\f \in W^{-1, p}_k(\Rdt)^\dimt$  satisfying
\begin{equation}\label{condition_compa31}
 \forall i \leq 3, \;  \< f_i, 1\> = 0. 
\end{equation}
and
\begin{equation}
\|\f \|_{W^{-1, p}_k(\Rdt)^\dimt} \lesssim \kappa,
\end{equation}
there exists a unique weak solution $(\u, \pi)  \in V^{1,p}_{k}(\Rdt)^\dimt \times W^{0,p}_{k}(\Rdt)$
of (\ref{prob_opT}) satisfying 
\begin{equation}
\|\u\| _{W^{1, p}_k(\Rdt)^\dimt} + \|\Rw \u\| _{W^{-1, p}_k(\Rdt)^\dimt} + \|\pi\| _{W^{0, p}_k(\Rdt)^\dimt}  \lesssim   \|\f \|_{W^{-1, p}_k(\Rdt)^\dimt}.  
\end{equation}
Moreover, if $\f \in W^{0, p}_{k+1}(\Rdt)^\dimt$, then $(\u, \pi)  \in V^{2,p}_{k+1}(\Rdt)^\dimt \times W^{1,p}_{k+1}(\Rdt)$ and
\begin{equation}\label{estim_ord2}
\|\u\| _{W^{2, p}_{k+1}(\Rdt)^\dimt} + \|\Rw \u\| _{W^{0, p}_{k+1}(\Rdt)^\dimt} + \|\pi\| _{W^{1, p}_{k+1}(\Rdt)^\dimt}  \lesssim   \|\f \|_{W^{0, p}_{k+1}(\Rdt)^\dimt}.  
\end{equation}
\end{theorem}
The proof of theorem \ref{main_result1} will be achieved through several steps. It is postponed to 
section  \ref{proof_mr_sec} hereafter. 
\begin{remark}
Condition \ref{condition_compa31} is necessary and follows from (\ref{condi_gradgrad})  with $\vphi = \e_i$, where
$(\e_1,\e_2, \e_3)$ is the canonical basis of $\Rdt$. Notice that  $1 \in W^{1, p'}_{-k}(\Rdt)$. 
\end{remark}
\begin{remark}
Theorem \ref{main_result1} can be  extended  by means of Mozzi-Chasles transformation (see, e. g.,  \cite{galdisilvestre})  to equations of the form (\ref{modifiedNS0}) with $\vv_\infty \ne 0$ and  $\vvom \times \vv_\infty = 0$. 
\end{remark}
\begin{corollary}
Under the same assumptions of theorem \ref{main_result1}, the solution $(\u, \pi)$ satisfies the following asymptotic property
\begin{equation}\label{asympto_proper1}
\dps{ \lim_{|\x| \fleche +\infty} |\x|^{2-\eps} \|u(|\x|, .)\|_{L^p(\S^2)} = 0 }, \\
\end{equation}
where $\S^2$ denotes the unit sphere of $\R^3$ and $\eps=\frac{3}{p'} - [\frac{3}{p'}] \in ]0, 1[$. Moreover, 
\begin{itemize}
\item If $\f \in W^{0, p}_{k+1}(\Rdt)^\dimt$ (with $k = [\frac{\dimt}{p'}]$), then 
\begin{equation}\label{asympto_proper2}
\dps{ \lim_{|\x| \fleche +\infty} |\x|^{3-\eps} \|\pi(|\x|, .)\|_{L^p(\S^2)} = 0 }, 
\end{equation}
\item If $p > 3$ then $(1+|\x|)^{1+3/p} \u \in C_b(\R^3)^3$ and 
$$
\lim_{|\x| \to +\infty} |\x|^{1+3/p} |\u(\x)| = 0. 
$$
\item If $p > 3$ and $\f \in W^{0, p}_{3}(\Rdt)^\dimt$, then 
$(1+|\x|)^{2+3/p} \pi \in C_b(\R^3)$ 
and 
\begin{equation}\label{asympto_proper3}
\lim_{|\x| \to +\infty} |\x|^{2+3/p}| |\pi(\x)| = 0. 
\end{equation}
\end{itemize}
\end{corollary}
The asymptotic properties (\ref{asympto_proper1})-(\ref{asympto_propert})    can be deduced from
the following property (see, e. g.,  \cite{alliot_these}, Proposition 3.8):  there exists a constant $c > 0$ such that 
for all $u \in W^1_\alpha(\Rdt)$ and $R > 0$, 
$$
|\x|^{(\alpha -1)p +3 }  \int_{\S^2} |u(|\x|, .)|^p d\sigma   \leq c \int_{|\y| > R} \rhry^{\alpha p} |\grad u|^{p} dy, \; \mbox{ for } |\x| > R, 
$$
and when $p > n$
$$
\lim_{ |\x| \to +\infty }  |\x|^{\alpha + \frac{3}{p} -1} |u(|\x|, .)| = 0. 
$$

This result provides a further refinement to the one proven by Galdi \cite{galdi2003}, in which it is shown that if $f$
satisfies the condition 
$$
(1+|x|)^3 |f| \in L^\infty(\R^3)^3, 
$$
 along with a few other assumptions, then 
 $$
(1+|x|) \u \in L^\infty(\R^3)^3, \; (1+|x|)^2 \grad u \in L^\infty(\R^3)^{9}, \; (1+|x|)^3 p \in L^\infty(\R^3). 
$$
Let us also note that the estimates  \eqref{estim_ord2} differ from those obtained in \cite{farwig8}  for the linear problem  with rotation.

\section{The proof of the main  result}
The purpose here is to prove the main theorem. Let us start with the following section, which contains some preliminary results useful for the proof.
\subsection{Some useful preliminary results on weighted spaces}\label{prepsec}
The results in this section are all stated in $\R^n$, where $n \geq 1$ is arbitrary and not necessarily equal to $3$.
We prove here for later use some results on weighted Sobolev spaces. 
We assume throughout the section that dimension $\dim\geq 1$ is arbitrary and is not necessarily equal to $3$.  \\
In the sequel, when $1< p < \dim$ define $\ps \in [1, \infty)$ as:
$$
\frac{1}{\ps} = \frac{1}{p}- \frac{1}{\dim}. 
$$
The following lemma is a direct consequence of H\"older inequality
\begin{lemma}\label{product_lemma0}
Let   $1 < p < +\infty$,   $1 < r < +\infty$  and $1 < s \leq  +\infty$  be  such that
$$
 \frac{1}{s} + \frac{1}{r} = \frac{1}{p}.
$$
Let $u \in W^{0, r}_\alpha(\Rd)$ and 
$v \in W^{0, s}_\beta(\Rd)$ with $\alpha, \beta \in \R$.  Then, $uv \in W^{0, p}_{\alpha+\beta}(\Rd)$.
\end{lemma}
\begin{lemma}\label{lemma_inj1}
Let $1 < p \leq + \infty$ and $\alpha \in \R$. We have
\begin{itemize}
\item If $p < \dim$,  then $W^{1, p}_\alpha(\Rd) \inclus W^{0, \ps}_\alpha(\Rd)$,
\item If $p > \dim$,  then $W^{1, p}_\alpha(\Rd) \inclus W^{0, \infty}_{\alpha+\dim/p-1}(\Rd)$. 
\end{itemize}
\end{lemma}
\begin{proof}
\begin{itemize}
\item  Assume that $p < \dim$. Sobolev's inequality states that
\begin{equation}\label{sobo_ineq}
\|\phi\|_{L^{\ps}(\Rd)}  \lesssim \|\grad \phi\|_{L^{p}(\Rd)}, 
\end{equation}
for $\phi \in \D (\Rd)$. By density of $\D (\Rd)$ in  $W^{1,p}_0(\Rd)$ this inequality remains valid for $\phi \in W^{1,p}_0(\Rd)$.
Let $u \in W^{1, p}_\alpha(\Rd)$. Taking $ \phi = \rhr^\alpha u \in W^{1,p}_0(\Rd)$ in  (\ref{sobo_ineq})  gives
$$
\|\rhr^{\alpha} u\|_{L^{\ps}(\Rd)} \lesssim  \|\grad (\rhr^\alpha u) \|_{L^{p}(\Rd)}\lesssim  \|u\|_{W^{1,p}_\alpha(\Rd)}. 
$$
\item Assume now that $ \dim < p < +\infty $. One has $W^{1, p}_\alpha(\Rd)  \inclus W^{1, p}_{loc}(\Rd) \inclus \C^0_{loc}(\Rd)$. 
The well known  Morrey's inequality says  that for any $ \varphi \in W_{loc}^{1, p}(\Rd)$
$$
|\varphi(\x)- \varphi(\y)|  \lesssim C |\x|^{1-\dim/p} \|\grad \varphi\|_{L^p(B(\x, 2 r))} 
$$
for all $\x, \y \in \Rd$ such that $|\x-\y| = r \geq 0$. Here    $B(\x, 2r) = \{ \z \in \Rd \bve |\x - \z| \leq 2r\}$
and $C$ is constant depending only on $p$ and $\dim$. The inequality remains valid when $\varphi \in W^{1, p}_0(\Rd)$.  \\
Let   $\chi$  be a fixed   ${\mathcal C}^{\infty}$ function satisfying 
 $$
 \chi(\x) =1 \mbox{ for  } |\x| \leq 1,  \; \chi(\x) =0 \mbox{ for  } |\x| \geq 2. 
 $$
Consider a function $v \in W^{1, p}_\alpha(\Rd)$.  One can write $v= v_1+v_2$ with  $v_1 = \chi v \in \C_b(\R^3)$ and $v_2 = (1-\chi)v \in W^{1,p}_\alpha(\Rd)$. 
Morrey inequality applied to $\rhr^\alpha v_2 \in W^{1,p}_0(\Rd)$, with $\y = \zero$,  gives 
$$
\begin{array}{rcl}
|\rhr^\alpha  v_2(\x)| & \lesssim& |\x|^{1-\dim/p} \|\grad (\rhr^\alpha v_2)\|_{L^p(\Rd)}  \\
 &  \lesssim&    |\x|^{1-\dim/p} \|v_2\|_{W^{1, p}_\alpha(\Rd)}  
\end{array}
$$
Thus, $\rhr^{\alpha + \dim/p -1} v_2 \in L^\infty(\Rd)$.  On the other hand,
$\rhr^{\alpha + \dim/p -1} v_1 \in L^\infty(\Rd)$  since $v_1(\x) = 0$ when $|\x| \geq 2$. This ends the proof. 
\end{itemize}
\end{proof}
\begin{corollary}\label{coro1_0}
Let $p$, $q$,   $\alpha$ and $\beta \in \R$ four real numbers such that $1 < p \leq +\infty$, $1 < q < +\infty$ and
\begin{equation}\label{cond_inclu}
 0 \leq \frac{\dim}{q} - \frac{\dim}{p}+1 =  \frac{\dim}{q}  -  \frac{\dim}{\ps}   < \alpha - \beta.
 \end{equation}
Then,   
\begin{equation}\label{nouv_incl}
W^{1, p}_\alpha(\Rd) \inclus W^{0, q}_\beta(\Rd).
 \end{equation}
 Moreover, this imbedding remains valid if $q =+\infty$, $p > n$ and $\alpha - \beta \geq 1 - \frac{\dim}{p}$. 
\end{corollary}
\begin{proof}
Three cases are distinguished 
\begin{itemize}
\item Case 1 :  $ \dim < p \leq +\infty$. The proof follows immediately from Lemma  \ref{lemma_inj1} and the imbedding
$ W^{0, \infty}_{\alpha+\dim/p-1}(\Rd) \inclus    W^{0, q}_\beta(\Rd)$ (which remains valid even if $q = +\infty)$). 
\item Case 2 :  $p < \dim$.   The left inequality in (\ref{cond_inclu})  means that $q \leq \ps$.  Combining  Lemma \ref{lemma_inj1} with
 \eqref{holder_weighted} gives 
$$
W^{1, p}_\alpha(\Rd)  \inclus W^{0, \ps}_\alpha(\Rd)  \inclus W^{0, q}_\beta(\Rd). 
$$
\item Case 3: $1 <  p = \dim$ and $q < +\infty$.  Let  $\eps  \in ]0, 1[$ picked so  that 
$$
 \frac{n}{q} + 1 > \frac{\dim}{n - \eps},  
$$
and set $p_\eps = p - \eps =  \dim - \eps < n$.   Let $\talpha$ be a real number chosen such that
$$
 \beta + \frac{\dim}{q} + 1 -  \frac{\dim}{p_\eps}  < \talpha  < \alpha + \frac{\dim}{p} -   \frac{\dim}{p_\eps}. 
$$ 
Under these constraints on $\eps$ and $\tilde{\alpha}$, we can apply \eqref{holder_weighted} and  \eqref{nouv_incl} , and we obtain the following inclusions: 
$$
W^{1, p}_\alpha(\Rd) \inclus W^{1, p_\eps}_{\talpha}(\Rd)\inclus W^{0, q}_{\beta}(\Rd). 
$$
\end{itemize}
\end{proof} 
\begin{corollary}
Let  $s > 1$,  $q >1$, $\alpha$ and $\beta$ be four reals and $m \geq 1$ be an integer. Assume that
\begin{equation}\label{cond_minclu}
\alpha  - \beta >  \frac{n}{q}-\frac{n}{s}+m  \geq  0.  
\end{equation}
Then, the following imbedding holds and is compact
\begin{equation}
W^{m, s}_\alpha(\Rd) \inclus W^{0, q}_\beta(\Rd).
\end{equation}
\end{corollary}
\begin{proof}
 Consider the finite sequences $(\theta_k)_{0 \leq k \leq m}$, $(s_k)_{0 \leq k \leq m}$ and $(\alpha_k)_{0 \leq k \leq m}$ defined as: 
$$
\theta_k = \frac{\dim}{q}+ k (1-\eps_1),  \; s_k = \frac{\dim}{\theta_k}, \; \alpha_k = \beta + k \eps_2, \; \mbox{ for } 0 \leq k \leq m. 
$$
with
$$
\eps_1 =\frac{1}{m}\left( \frac{\dim}{q}  - \frac{\dim}{s} + m\right), \; \eps_2 = \frac{\alpha-\beta}{m} > 0. 
$$
We have $\theta_m = \dim/s$,  $\alpha_m = \alpha$,  $\eps_2 >\eps_1 \geq  0$, thanks to condition (\ref{cond_minclu}).
We also have for $1 \leq k \leq m$ 
$$
0< \frac{\dim}{s_{k-1}} - \frac{\dim}{s_{k}} + 1 = \theta_{k-1} - \theta_k + 1 = \eps_1 < \alpha_k-\alpha_{k-1} = \eps_2. 
$$
We also have $1 < s_0 = q< ...<s_m = s$ or  $1 < s_m= s< ...<s_0 = q$.  Thus, 
$$
W^{m, s_m}_{\alpha_m}(\Rd) = W^{m, s}_{\alpha}(\Rd)    \inclus \cdots \inclus W^{k, s_k}_{\alpha_k}(\Rd) \inclus \cdots \inclus   W^{0, s_0}_{\alpha_0}(\Rd) = W^{0, q}_{\beta}(\Rd). 
$$
\end{proof}

\begin{lemma}\label{product_lemma}  
Let  $1 < p < +\infty$,  $v \in W^{1, p}_\alpha(\Rd)$ and  $w \in W^{0, p}_\alpha(\Rd)$.
\begin{itemize} 
\item If  $p > \dim$. Then, $v w \in W^{0, p}_{2\alpha + \dim/p-1}(\Rd)$.
 \item If  $\max(\frac{\dim}{2} , \frac{2 \dim}{\dim + 1}) < p < \dim$  then $v w  \in W^{0, s}_{2\alpha}(\Rd)$ where $s > 1$ is defined as $\frac{1}{s} = \frac{2}{p} - \frac{1}{\dim}$. 
 \end{itemize}
 In both the cases, if $\alpha > 2 - \dim/p$ then  $v w  \in W^{-1, p}_{\alpha}(\Rd)$ and 
 \begin{equation}
\|v w\|_{W^{-1, p}_\alpha(\Rd)} \lesssim \|v\|_{W^{1, p}_\alpha(\Rd)} \|w\|_{W^{0, p}_\alpha(\Rd)}.
\end{equation}
\end{lemma}
\begin{proof}
\begin{itemize}
\item If $p > \dim$, then $v \in  W^{0, \infty}_{\lambda}(\Rd)$ with $\lambda=  \alpha + \dim/p-1$, thanks to Lemma  \ref{lemma_inj1}. Thus, in view of Lemma \ref{product_lemma0}
we have $ v w \in  W^{0, p}_{2\alpha + \dim/p-1}(\Rd)$.  If $\alpha > 2 - \dim/p$, $W^{0, p}_{2\alpha + \dim/p-1}(\Rd) \inclus W^{0, p}_{\alpha+1}(\Rd) \inclus W^{-1, p}_{\alpha}(\Rd)$. 
\item If   $\dim/2 < p < \dim$, then from Lemma \ref{lemma_inj1} we know that  $v \in W^{0, \ps}_\alpha(\Rd)$. 
By H\"older inequality,  $vw \in W^{0, s}_{2\alpha}(\Rd)$ where  
$$
\frac{1}{s} = \frac{1}{p} + \frac{1}{\ps} =  \frac{2}{p} - \frac{1}{\dim}.
$$
If  $\alpha > 2 - \dim/p$, then   $W^{1, p'}_{-\alpha}(\Rd) \inclus W^{0, s'}_{-2\alpha}(\Rd)$, thanks to Corollary \ref{coro1_0}. Hence, 
$W^{0, s}_{2\alpha}(\Rd) \inclus W^{-1, p}_{\alpha}(\Rd)$. This ends the proof. 
\end{itemize}
\end{proof}
\subsection{Existence, uniqueness and regularity of weak solutions}\label{proof_mr_sec}
 Throughout the remainder of this work, we will restrict ourselves to the three-dimensional case, which is the only one relevant here. Accordingly, 
 we will assume from now on that $n= 3$. Hereafter we focus our intention of the proof of theorem \ref{main_result1}. It will be achieved by means of 
a fixed point argument. We need the following result concerning the linearized equations  
 (see Abada et {\it al.} \cite{abada_boulmezaoud}): 
\begin{equation}\label{linearNS}
\begin{array}{rcl}
-   \Delta \vv + \Rw \vv + \grad r  &=& \h \mbox{ in } \Rdt, \\
\div \vv &=& 0 \mbox{ in } \Rdt. 
\end{array}
\end{equation}
\begin{proposition}[\cite{abada_boulmezaoud}] \label{proposition3}
 Let  $p > 1$ be a real such that $ p \not \in \{\frac{3}{2}, 3\}$ and  $\ell$ an integer satisfying
 $$
\ell =   \left[\frac{\dimt}{p'}\right].  
  $$
  Let $\h \in W^{-1, p}_\ell(\Rdt)^\dimt$   such that 
\begin{equation}\label{condition_compa32}
  \< h_i, 1\> = 0 \mbox{ for } 1 \leq i \leq 3. 
\end{equation}
Then, the system  (\ref{linearNS}) has a unique solution $(\vv, r) \in V^{1,p}_{\ell}(\Rdt)^\dimt \times W^{0,p}_{\ell}(\Rdt)$ and 
\begin{equation}
 \|\vv \|_{W^{1, p}_\ell(\Rdt)^\dimt}  + \| \Rw \vv \|_{W^{-1, p}_\ell(\Rdt)^\dimt }  +    \|r\|_{W^{0, p}_\ell(\Rdt)} 
\lesssim \|\h\|_{W^{-1, p}_\ell(\Rdt)^\dimt}. 
\end{equation} 
Moreover, if $\h \in W^{0, p}_{\ell+1}(\Rdt)^\dimt$ 
then $(\vv, r) \in V^{2,p}_{\ell+1}(\Rdt)^\dimt \times W^{1,p}_{\ell+1}(\Rdt)$. 
\end{proposition}
Observe now that $k = [\dimt/p']$ satisfies the following inequaliy
$$
k > 2 - \frac{3}{p} = 2 - \frac{3}{p} = \frac{3}{p'} - 1, 
$$
since $p' \not \in  \{3/2, 3\}$.  \\
We are now in position to prove Theorem \ref{main_result1}. We start with existence and uniqueness. 
\begin{itemize}
\item {\it  Existence and uniquness}. 
Consider the spaces
$$
\begin{array}{rcl}
V &=& \{ \h \in W^{-1, p}_{k}(\Rdt)^\dimt  \bve    \< h_i, 1\> = 0 \; \mbox{ for } 1 \leq i \leq 3\}, \\
H &=& \{ \vv \in W^{1, p}_{k}(\Rdt)^\dimt \bve \div \vv = 0 \}. 
\end{array}
$$
According to Proposition  \ref{proposition3}, given a function $\h \in V$, there exists one and only one  pair $(\vv, r) \in   H \times W^{0, p}_k(\Rdt)$ solution of 
(\ref{linearNS}). 
Thus, one can  define  the linear  operator 
$$
\begin{array}{rcrl}
T\; &:&\; V &\rightarrow H, \\
&& \h& \mapsto \vv \mbox{ with $(\vv, r)$  the unique solution  of ( \ref{linearNS}) in which $\h$ is the RHS. } 
\end{array}
$$
There exists a constant $c_1 > 0$ such that
\begin{equation}\label{contT_inequa1}
\forall \h \in V, \; \|T(\h)\|_{W^{1, p}_{k}(\Rdt)}  \leq c_1  \|\h\|_{W^{-1,  p}_{k}(\Rdt)}. 
\end{equation}
On the other hand, given  $\vv, \; \z \in H$, one has $\vv. \grad \z \in W^{-1, p}_{k}(\Rdt)^3$, thanks to  Lemma \ref{product_lemma}, and 
\begin{equation}\label{contProd_inequa1}
\forall \vv \in H, \; \|\vv.\grad \z\|_{W^{-1, p}_{k}(\Rdt)}  \leq c_2\|\vv\|_{W^{1,  p}_{k}(\Rdt)} .   \|\z\|_{W^{1,  p}_{k}(\Rdt)}, 
\end{equation}
for some constant $c_2 > 0$ not depending on $\vv$ nor on $\z$.  Moreover, one has
$$
\int_{\Rdt} \vv.\grad z_i dx = - \int_{\Rdt} (\div  \vv ) z _i dx =  0. 
$$
In other terms, $\vv . \grad \z \in V$ when $\vv \in H$ and $\z \in H$. \\
Consider now the map $K \; :\; H \fleche H$ defined as
$$
\forall \vv \in H, \; K \vv = T(\f - \vv.\grad\vv). 
$$
The problem (\ref{prob_opT}) consists to find a fixed point of the operator $K$.  
\begin{lemma}\label{estimatesK}
The following estimate hold for all $\vv, \; \w \in H$
$$    
\begin{array}{rcl}
\; \|K(\vv)\|_{W^{1, p}_{k}(\Rdt)^\dimt} &\leq& c_1 \dps{  \{ \|\f\|_{W^{-1, p}_{k}(\Rdt)^\dimt} + c_2 \|\vv\|^2_{W^{1, p}_{k}(\Rdt)^\dimt} \}},   \\
\; \| K(\vv) - K(\w)\|_{ W^{1, p}_k(\Rdt)^\dimt  } & \leq & c_1 c_2 (\|\vv\|_{W^{1, p}_{k}(\Rdt)^\dimt}+ \|\w\|_{W^{1, p}_{k}(\Rdt)^\dimt} ) \|\w-\vv\|_{W^{1, p}_{k}(\Rdt)},
\end{array}        
$$
where $c_1 > 0$ and $c_2 > 0$ are the  constants in (\ref{contT_inequa1}) and (\ref{contProd_inequa1}). 
\end{lemma}
\begin{proof}
The first inequality follows from (\ref{contT_inequa1})  and (\ref{contProd_inequa1}). 
Let   $\u_1 = T(\vv), \; \u_2 = T(\w)$, and $\pi_1$, $\pi_2$ the corresponding pressures. The pair
$(\z, e) = (\u_2-\u_1, \pi_2 - \pi_1) \in H \times W^{0, p}_k(\Rdt)$ satisfies 
\begin{equation}
-\Delta \z -(\uom \times \x).\nabla \z +\uom\times \z +\nabla e = \vv \grad \vv -\w \grad \w. 
\end{equation}
Thus,  
$$
\begin{array}{rcl}
\|\z\|_{W^{1, p}_k(\Rdt)}  &\leq& \dps{  c_1 \|\w . \nabla \w-\vv . \nabla \vv\|_{W^{-1,p}_k(\Rdt)^\dimt} } \\
&\leq& \dps{  2 c_1 \|\vphi . \nabla \vtheta + \vtheta . \grad \vphi\|_{W^{-1,p}_k(\Rdt)^\dimt} } \\
\end{array}
$$
where we set
$$
\vphi = \frac{1}{2} (\vv + \w), \; \vtheta = \frac{1}{2}(\w-\vv). 
$$
It follows that 
$$
\begin{array}{rcl}
\|\z\|_{W^{1, p}_k(\Rdt)^\dimt}&\leq & \dps{  2 c_1\{ \| \vphi.\grad \vtheta\|_{W^{-1, p}_k(\Rdt)^\dimt} +  \| \vtheta.\grad \vphi\|_{W^{-1,p}_k(\Rdt)^\dimt} \} } \\
&\leq & \dps{ 4 c_1 c_2 \| \vphi\|_{W^{1, p}_k(\Rdt)^\dimt}   \| \vtheta\|_{W^{1, p}_k(\Rdt)^\dimt}}.
\end{array}
$$
This ends the proof of Lemma \ref{estimatesK}. 
\end{proof}
We pursue the proof of Theorem \ref{main_result1}. Assume
$$
\|\f\|_{W^{-1, p}_{k}(\Rdt)} < \frac{1}{4 c_1^2 c_2},
$$
and consider  the set
$$
\Ball = \{ \vv \in H \bve \|\vv\|_{W^{1, p}_k(\Rdt)^\dimt} \leq \rho \}, 
$$
where
$$
\rho = 2 c_1 \|\f\|_{W^{-1, p}_{k}(\Rdt)}.
$$
Then, for all $\vv \in \Ball$ 
$$
\begin{array}{rcl}
\|K \vv\|_{W^{1, p}_{k}(\Rdt)} & \leq&  \dps{ c_1  \|\f\|_{W^{-1, p}_{k}(\Rdt)} + c_1 c_2 \rho^2 }  \\
&\leq & \dps{ c_1  \|\f\|_{W^{-1, p}_{k}(\Rdt)} (1 + 4c^2_1 c_2  \|\f\|_{W^{-1, p}_{k} (\Rdt)} ) }  \\
&\leq & \rho. 
\end{array}
$$
Further, for all $\vv, \, \w \in \Ball$
$$
 \| K(\vv) - K(\w)\|_{W^{1, p}_\alpha(\Rdt)^\dimt  }  \leq 2 c_1 c_2  \rho  \|\w-\vv\|_{W^{1, p}_{k}(\Rdt)^\dimt}.
$$
It follows that $K$ is contractive since
$$
2 c_1 c_2 \rho = 4 c^2_1 c_2 \|\f\|_{W^{-1, p}_{k}(\Rdt)} < 1. 
$$
Applying Banach fixed point theorem, one deduces that there exists a unique pair $(\u, \pi) \in H \times W^{0, p}_k(\Rdt)$ solution of
(\ref{prob_opT}) and satisfying the estimate
$$
\|\u\|_{W^{1, p}_k(\Rdt)^\dimt} \leq  2 c_1 \|\f\|_{W^{-1, p}_{k}(\Rdt)}, \; \|\pi\|_{W^{0, p}_k(\Rdt)^\dimt} \leq c_3 \|\f\|_{W^{-1, p}_{k}(\Rdt)}. 
$$ 
\item {\it Regularity}.  Assume that  $\f \in W^{0, p}_{k+1}(\Rdt)$. Two cases are distinguished.
\begin{itemize}
 \item {\it Case 1: $p > \dimt$}.   Then, $k=2$ and  $\u.\grad \u \in  W^{0, p}_{3+3/p}(\Rdt)^\dimt \inclus W^{0, p}_{3}(\Rdt)^\dimt = W^{0, p}_{k+1}(\Rdt)^\dimt$, thanks to Lemma \ref{product_lemma},  and  $\f - \u.\grad\u \in W^{0, p}_{k+1}(\Rdt)^3$. 
 From Proposition \ref{proposition3}, we  conclude that  $\u = T(\f-\u.\grad \u) \in  W^{2, p}_{k+1}(\Rdt)^3$. \\
  \item {\it Case 2: $ \dimt/2  < p < \dimt$}. Then,  $k=[3/p']=1$.  According to  Lemma \ref{product_lemma}, we know that $\u.\grad \u \in W^{0, s}_{2k}(\Rdt)^3$ with 
  $$
  \frac{1}{s} = \frac{2}{p} - \frac{1}{\dimt}
  $$
Define $r > 1$ by 
$$
\frac{2}{r} = \frac{1}{\dimt} +  \frac{1}{p}. 
$$
Then $p < r < \dimt$ and $W^{1, r'}_{-1}(\Rdt) \inclus  W^{0,p'}_{-k-1}(\Rdt)  \cap  W^{0, s'}_{-2k}(\Rdt) = W^{0,p'}_{-2}(\Rdt)  \cap  W^{0, s'}_{-2}(\Rdt)$, thanks to 
Lemma \ref{coro1_0}. Thus, $W^{1,p}_{k+1}(\Rdt)  \inclus W^{-1, r}_{1}(\Rdt) $ and \\ $W^{0, s}_{2k}(\Rdt) \inclus W^{-1, r}_{1}(\Rdt)$. Hence,  $\f - \u.\grad \u \in  W^{-1,r}_{1}(\Rdt)^\dimt$. From  Proposition \ref{proposition3} (see also \cite{abada_boulmezaoud}) we know that  $\u =  T(\f-\u.\grad \u) \in  W^{1, r}_{1}(\Rdt)^3$. By Lemma \ref{product_lemma} again, we deduce that $u.\grad \u \in W^{0, p}_{2}(\Rdt)^\dimt $. Thus,    $\f - \u.\grad \u \in  W^{0,p}_{2}(\Rdt)^\dimt$ and, 
  in view of  Proposition \ref{proposition3},  $\u =  T(\f-\u.\grad \u) \in  W^{2, p}_{k+1}(\Rdt)^\dimt $. 
\end{itemize}
\end{itemize}
\section{Conclusion}
The use of weighted Sobolev spaces to describe the behavior of the velocity and pressure in the equations governing stationary 
fluid flows has once again proved to be a fruitful approach, despite the nonlinear nature of the equations studied here. 
The presence of terms arising from the rotation of the body significantly complicates the analysis of the system, both
 in the linear and nonlinear settings. Nevertheless, by combining the results obtained in the linear case with a fixed-point
  argument, we are able to successfully complete the existence theory for the full problem. Moreover, this approach provides 
  a fairly precise description of the asymptotic decay at infinity of the solutions.

$\;$\\
{\bf Declarations}. \\
{\it Conflict of interest}: The author declares no competing interests.

\bibliographystyle{plain} 
\bibliography{biblio_bkk_2025}

\end{document}